\documentclass[11pt]{article}
\usepackage [dvips]{graphics}
\usepackage[centertags]{amsmath}
\usepackage{amsfonts}
\usepackage{amssymb}
\usepackage{amsthm}
\usepackage{newlfont}
\usepackage{stmaryrd}
\usepackage{color,epsfig}
\usepackage{amsfonts,graphicx,psfrag}
\usepackage{graphicx}
\usepackage{float}
\usepackage{amsmath}

\usepackage{url}

\usepackage{txfonts}
\usepackage{subcaption}
\usepackage{mathrsfs}
\usepackage{verbatim}

\theoremstyle{plain}
\numberwithin{equation}{section}
\newtheorem{thm}{Theorem}[section]
\newtheorem{cor}[thm]{Corollary}
\newtheorem{lem}[thm]{Lemma}

\newtheorem{rem}[thm]{Remark}

\def\sqr#1#2{{\vcenter{\vbox{\hrule height.#2pt
				\hbox{\vrule width.#2pt height#1pt \kern#1pt \vrule
					width.#2pt}
				\hrule height.#2pt}}}}

\def\<{{\langle}}
\def\>{{\rangle}}
\def\no{\noindent}

\def\bs{\bigskip}

\def\bea{\begin{eqnarray}}
\def\eea{\end{eqnarray}}
\def\bes{\begin{eqnarray*}}
\def\ees{\end{eqnarray*}}

\makeatletter

\@addtoreset{equation}{section}
\makeatother

\begin{document}
%	\begin{frontmatter}
		\title{\bf Three elliptic closed characteristics on the non-degenerate compact convex hypersurfaces in $\mathbb{R}^{6}$ }
		\author{Lu Liu \thanks{E-mail:202211806@mail.sdu.edu.cn.}
			\quad Yuwei Ou \thanks{E-mail:ywou@sdu.edu.cn.}
			\\ \\
			Department of Mathematics, Shandong University Jinan, Shandong 250100 \\
			The People's Republic of China
		}
		\date{}
		\maketitle
				
		\begin{abstract}
	Let $\Sigma\subset \mathbb{R}^{2n}$ with $n\geq2$ be any $C^2$ compact convex hypersurface. The stability of closed characteristics has attracted considerable attention in related research fields. A long-standing conjecture states that all closed characteristics are irrationally elliptic, provided $\Sigma$ possesses only finitely geometrically distinct closed characteristics. This conjecture has been fully resolved only in $\mathbb{R}^4$, while it remains completely open in higher dimensions. Even in $\mathbb{R}^6$, it is unknown whether there exist three elliptic closed characteristics. In this paper, we first prove that for any $\Sigma\subset \mathbb{R}^{2n}$ with finitely many closed characteristics, there exist at least two elliptic closed characteristics, which possess a nice symplectic normal form. In particular, as a simple corollary, they are irrational elliptic when $\Sigma$  is non-degenerate. Moreover, for any non-degenerate $\Sigma\subset\mathbb{R}^{6}$ with finitely many closed characteristics, we obtain at least three elliptic characteristics, of which at least two are irrationally elliptic. Based on the $n$-or-$\infty$ conjecture, three elliptic closed characteristics are optimal.
This result provide theoretical support for further research on this conjecture.

		\end{abstract}		
	\bs	
	
	\no{\bf AMS Subject Classification:} 58E05, 53D12,   34C25
		
	\bs	
	
	\no{\bf Key Words:}  Compact convex hypersurfaces, Closed characteristics, Stability, Maslov-type index, Hamiltonian systems		
%	\end{frontmatter}		
		\section{Introduction and main results}
		Let \(\Sigma\) be a $C^2$ compact hypersurface in \(\mathbb{R}^{2n}\) that is strictly convex with respect to the origin, i.e., the tangent hyperplane of \(\Sigma\) at any \(x \in \Sigma\) does not pass through the origin. Without loss of generality, we suppose it contains the origin. We denote the set of all such hypersurfaces in \(\mathbb{R}^{2n}\) by \(\mathcal{H}(2n)\).
		
		For \(x \in \Sigma\), let \(N_{\Sigma}(x)\) be the outward unit normal vector of \(\Sigma\) at \(x\), satisfying \(N_{\Sigma}(x) \cdot x = 1\).  We consider closed characteristics $(\tau, x)$ on $\Sigma$, which are solutions of the following problem
	\begin{equation}
		\begin{cases}
			\dot{x}(t) = JN_{\Sigma}(x(t)),\ x(t) \in \Sigma,\ \forall t \in \mathbb{R}, \\
			x(\tau) = x(0),
		\end{cases}
		\label{system 1}
	\end{equation}
		where \(J = \begin{pmatrix} 0 & -I_n \\ I_n & 0 \end{pmatrix}\) is the standard symplectic matrix and \(\tau > 0\). A closed characteristic \((\tau, x )\) is prime with \(\tau\) being the minimal period. Two closed characteristics \((\tau, x)\) and \((\sigma, y)\) are called geometrically distinct if \(x(\mathbb{R}) \neq y(\mathbb{R})\). The set of all geometrically distinct closed characteristics is denoted by \(\widetilde{\mathcal{J}}(\Sigma)\), and let \([(\tau, x)]\) be the equivalence class of closed characteristics geometrically identical to \((\tau, x)\).
			
		Let \( j : \mathbb{R}^{2n} \to \mathbb{R} \) be the gauge function of \( \Sigma \), i.e., \( j(\lambda x) = \lambda \) for \( x \in \Sigma \) and \( \lambda \geq 0 \), then \( j \in C^2(\mathbb{R}^{2n} \setminus \{0\}, \mathbb{R}) \cap C^1(\mathbb{R}^{2n}, \mathbb{R}) \) and \( \Sigma = j^{-1}(1) \). Fix a constant \( \alpha \in (1,2) \) and define the Hamiltonian function \( H_{\alpha} : \mathbb{R}^{2n} \to [0,+\infty] \) by
		\begin{equation*}
			H_{\alpha}(x) = j(x)^{\alpha}, \quad \forall x \in \mathbb{R}^{2n}.
		\end{equation*}
		Then \( H_{\alpha} \in C^2(\mathbb{R}^{2n} \setminus \{0\}, \mathbb{R}) \cap C^1(\mathbb{R}^{2n}, \mathbb{R}) \) is convex and \( \Sigma = H_{\alpha}^{-1}(1) \). It is well known that the problem \eqref{system 1} is equivalent to the following given energy problem of the Hamiltonian system
		\begin{equation}
			\begin{cases}
				\dot{x}(t) = J H_{\alpha}'(x(t)),\ H_{\alpha}(x(t)) = 1, & \forall t \in \mathbb{R}, \\
				x(\tau) = x(0).
			\end{cases}
			\label{system 2}
		\end{equation}
		Denote by \( \widetilde{\mathcal{J}}(\Sigma, \alpha) \) the set of all geometrically distinct solutions \( (\tau, x) \) of \eqref{system 2} where \( \tau \) is the minimal period of \( x \). Note that elements in \( \widetilde{\mathcal{J}}(\Sigma) \) and \( \widetilde{\mathcal{J}}(\Sigma, \alpha) \) are one to one correspondent to each other.

		Let \( (\tau, x) \in \widetilde{\mathcal{J}}(\Sigma, \alpha) \), the fundamental solution \( \gamma_x : [0,\tau] \to \text{Sp}(2n) \) with \( \gamma_x(0) = I_{2n} \) of the linearized Hamiltonian system
		\begin{equation}
			\dot{y}(t) = J H_{\alpha}''(x(t)) y(t), \quad \forall t \in \mathbb{R}
			\label{linear system 3}
		\end{equation}
		is called the associated symplectic path of \( (\tau, x) \). The eigenvalues of \( \gamma_x(\tau) \) are called Floquet multipliers, and
denote the elliptic height $e(\gamma_x(\tau))$ of $\gamma_x(\tau)$ to be the total algebraic multiplicity of all eigenvalues of $\gamma_x(\tau)$ on the unit circle
$\mathbb{U}=\{z \in \mathbb{C}\, |\, |z|=1 \}$. As usual, a closed characteristic $(\tau, x)$ is non-degenerate, if $1 \in \sigma(x)$ is of algebraic multiplicity $2$; a hypersurface $\Sigma$ is called non-degenerate, if all the closed characteristics on the given hypersurface together with all of their iterations are non-degenerate as periodic
solutions of the corresponding Hamiltonian system. A closed characteristic $(\tau, x)$ is called elliptic, if $e(\gamma_x(\tau))=2n$; hyperbolic if $1$ is a double Floquet multiplier of it and $e(\gamma_x(\tau))=2$; irrationally elliptic if in an appropriate coordinate $\gamma_x(\tau)$ is the $\diamond$-product of one $N_{1}(1,1)=\left(
                                                                                       \begin{array}{cc}
                                                                                         1 & 1 \\
                                                                                         0 & 1 \\
                                                                                       \end{array}
                                                                                     \right)
$
and $n-1$ rotation matrices $\left(
                               \begin{array}{cc}
                                 \cos\theta_{i} & -\sin\theta_{i} \\
                                 \sin\theta_{i} & \cos\theta_{i} \\
                               \end{array}
                             \right)
$
with rotation angles $\theta_{i}$ being irrational multiples of $\pi$ for $1\leq i\leq n-1$. Hence, the stability problem refers to estimating the number of elliptic closed characteristics that may exist for system \eqref{linear system 3}.
	
		%%文献综述
		Research on the stability of closed characteristics dates back to 1986, I. Ekeland \cite{eke1986} proved that if \(\Sigma \in \mathcal{H}(2n)\) is a \(\sqrt{2}\)-pinched compact convex hypersurface, then it admits at least one elliptic closed characteristic. Later, he further made following conjectured in \cite{eke1990}:

\vskip 0.2 cm
\textbf{Conjecture 1}. For any \(\Sigma \in \mathcal{H}(2n)\), there exists at least one elliptic closed characteristic.
\vskip 0.2 cm
In \cite{dde1992}, Dell'Antonio-D'Onofrio-Ekeland verified the above conjecture by imposing the symmetry condition \(\Sigma = -\Sigma\). In 1998, Long \cite{lon1998} showed that for any \(\Sigma \in \mathcal{H}(2n)\), either there exist infinitely many geometrically distinct hyperbolic closed characteristics \([(\tau_j, x_j)]\) with minimal periods \(\tau_j \to +\infty\), or there exists at least one non-hyperbolic closed characteristic. In 2002, Long and Zhu \cite{lonzhu2002} refined this result: If \(^{\#}\widetilde{\mathcal{J}}(\Sigma) < +\infty\), then there has at least one elliptic closed characteristic on \(\Sigma\), and at least \( [ \frac{n}{2} ] \) geometrically distinct closed characteristics possess irrational mean indices, where $^{\#}$A denotes the total number of elements in a set A. They further proved that if \(^{\#}\widetilde{\mathcal{J}}(\Sigma) \leq 2\rho_n(\Sigma) - 2 < +\infty\), then \(\Sigma\) admits at least two elliptic closed characteristics. Also, in 2000, Long \cite{lon2000} proved that for \(n = 2\), if \(^{\#}\widetilde{\mathcal{J}}(\Sigma) = 2\), then both closed characteristics are elliptic. Later, in 2007, \cite{wanghulong2007} strengthened this conclusion by showing that they are in fact irrational elliptic closed characteristics. Based on these results, when \(^{\#}\widetilde{\mathcal{J}}(\Sigma)<+\infty\), it is natural to guess that we should have a stronger result than the conclusion of Conjecture 1. Therefore, in \cite{wanghulong2007}, the authors further made the following conjecture:
\vskip 0.2 cm
\textbf{Conjecture 2}. For any \(\Sigma \in \mathcal{H}(2n)\), assume \(^{\#}\widetilde{\mathcal{J}}(\Sigma) < +\infty\),
then all the closed characteristics on $\Sigma$ are irrationally elliptic.
\vskip 0.2 cm
A typical example is the non-resonant  ellipsoid in $\mathbb{R}^{2n}$, that is $\Sigma$ is defined by
\bes \sum_{i=1}^n \frac{\alpha_i}{2} (p_i^2+q_i^2)=1, \ees
where $\alpha_i/\alpha_j\in\mathbb{R}\setminus \mathbb{Q}$. There just exist $n$  closed characteristics $x_i,i=1,...,n$ and  their mean Maslov-type index  satisfy $\hat{i}(x_i)/\hat{i}(x_j)=\alpha_j/\alpha_i\in\mathbb{R}\setminus\mathbb{Q}$ and all the closed characteristics
are irrationally elliptic. Regarding the progress on Conjecture 2, in 2006, Long-Wang \cite{lonwang2006} proved that for \(n = 3\), if \(^{\#}\widetilde{\mathcal{J}}(\Sigma) = 3\), then at least two of the closed characteristics are elliptic. In 2014, Wang \cite{wang2014} further upgraded this result to show that there exist at least two irrational elliptic closed characteristics. In 2022, Wang \cite{wang2022a} established that if \(\rho_n(\Sigma) = n\) and \(^{\#}\widetilde{\mathcal{J}}(\Sigma) = n\), then at least two of them are irrational elliptic closed characteristics.
Recently, Li-Liu-Wang \cite{liliuwang2025} improves \cite{wang2022a} by eliminating the technical condition \(\rho_n(\Sigma) = n\) and prove that if \(^{\#}\widetilde{\mathcal{J}}(\Sigma) = n\), then at least two of them are irrational elliptic closed characteristics. For results on more general compact star-shaped hypersurfaces, we refer the reader to \cite{liulong1999}, \cite{hulong2002}, \cite{liulong2015}, \cite{liliuwang2025}, etc.

\begin{rem}
In general, when \(^{\#}\widetilde{\mathcal{J}}(\Sigma) < +\infty\), it hard to estimate the upper bound of \(^{\#}\widetilde{\mathcal{J}}(\Sigma)\). This relate to the
following conjecture.
\vskip 0.2 cm
\textbf{The $n$-or-$\infty$ conjecture}: For any \(\Sigma \in \mathcal{H}(2n)\), then \(^{\#}\widetilde{\mathcal{J}}(\Sigma)=n\) or $\infty$.
\vskip 0.2 cm
In \cite{HWZ1998} of 1998, Hofer-Wysocki-Zehnder proved that \(^{\#}\widetilde{\mathcal{J}}(\Sigma) =2\) or $\infty$ holds
for every \(\Sigma \in \mathcal{H}(4)\). Furthermore, Colin-Dehornoy-Rechtman \cite{CDR2023} proved that there are two
or infinitely many simple Reeb orbits for any non-degenerate contact form on a closed connected three-manifold.
Recently, this conjecture was confirmed in \cite{CHHH2025}, by proving a stronger theorem that the two or infinity result holds for any closed contact 3-manifold whose contact structure possessing torsion first Chern class. However, for $n\geq 3$, this conjecture is still widely open. It is worth mentioning that recent progress by \c{C}ineli-Ginzburg-G\"{u}rel in \cite{CGG2024} of 2024 confirmed that \(^{\#}\widetilde{\mathcal{J}}(\Sigma)\geq n\) for $\Sigma$
bounding star-shaped domains with dynamically convex Reeb flows.
%\vskip 0.2 cm
%\textbf{Conjecture 3}. For any \(\Sigma \in \mathcal{H}(2n)\), \(^{\#}\widetilde{\mathcal{J}}(\Sigma)=n\) or $\infty$.
\end{rem}
Since the upper bound of \(^{\#}\widetilde{\mathcal{J}}(\Sigma)\) is hard to estimate, as it involves the challenging the $n$-or-$\infty$ conjecture. In 2017, under the original finiteness condition in Conjecture 2, i.e \(^{\#}\widetilde{\mathcal{J}}(\Sigma) < +\infty\), Hu-Ou \cite{huou2017} proved that for any \(\Sigma \in \mathcal{H}(2n)\), there has at least two elliptic closed characteristics on $\Sigma$. In this paper, we further obtain the following conclusions.
		\begin{thm}
			 For any surfaces $\Sigma \in \mathcal{H}(2n)$ with \(^{\#}\widetilde{\mathcal{J}}(\Sigma) < +\infty\), then at least two closed characteristics are elliptic whose momodormy matrix possess a nice symplectic normal form
\begin{equation}
		\begin{aligned}
			\gamma(\tau) &\simeq N_1(1,1)\diamond I_{2p_0} \diamond N_1(1,-1)^{\diamond p_+} \diamond N_1(-1,1)^{\diamond q_-} \diamond -I_{2q_0}\\
			&\ \ \diamond R(\theta_1) \diamond \cdots \diamond R(\theta_r)
			\diamond N_2(\lambda_1, v_1) \diamond \cdots \diamond N_2(\lambda_{r_0}, v_{r_0}),\nonumber
		\end{aligned}
	\end{equation}
with all $N_{2}(\lambda_{1},\nu_{1}),\ldots,N_{2}(\lambda_{r_{0}},\nu_{r_{0}})$ are rational normal form. Moreover, when $\Sigma$ is non-degenerate, then they are irrationally elliptic.
			%Suppose $^{\#}\widetilde{\mathcal{J}}(\Sigma) = 3$ for some non-degenerate $\Sigma \in \mathcal{H}(6)$, then all three closed characteristics on $\Sigma$ are elliptic.
			\label{thm:2-irr-elliptic}
		\end{thm}
\begin{rem}
By combining the methods of \cite{CGG2024}, \cite{liliuwang2025} and this paper, the result of Theorem \ref{thm:2-irr-elliptic} can perhaps be generalized to dynamically convex Reeb flow on star-shaped surfaces.
\end{rem}
\begin{rem}
How to obtain the irrationally elliptic without the non-degeneracy condition in Theorem \ref{thm:2-irr-elliptic} is an interesting question, but for now, we are unable to do so. In fact, it is still unknown whether the following \textbf{Assumption A} holds, except for the case $n=2$. We believe it to be true for $n\geq3$, but we are currently unable to prove it; perhaps a deeper analysis of the methods in \cite{liliuwang2025} might lead to progress.
\end{rem}
\textbf{Assumption A}: Consider the surfaces $\Sigma \in \mathcal{H}(2n)$ with \(^{\#}\widetilde{\mathcal{J}}(\Sigma) < +\infty\), then for any $a\in A(v)$, $x_{j(1)}$ is an irrational
elliptic closed characteristics, where $A(v)$ and $x_{j(1)}$ is given by Theorem \ref{thm:index jump 3}.
	\begin{thm}
			 Under Assumption A, then for any surfaces $\Sigma \in \mathcal{H}(2n)$ with \(^{\#}\widetilde{\mathcal{J}}(\Sigma) < +\infty\), at least two closed characteristics are irrational elliptic.
			%Suppose $^{\#}\widetilde{\mathcal{J}}(\Sigma) = 3$ for some non-degenerate $\Sigma \in \mathcal{H}(6)$, then all three closed characteristics on $\Sigma$ are elliptic.
			\label{thm:3-irr-elliptic}
		\end{thm}
So far, Conjecture 2 has been fully resolved only in $\mathbb{R}^4$, while it remains entirely an open problem in higher dimensions. Even in $\mathbb{R}^6$, it is unknown whether there exist three elliptic closed characteristics. Therefore, we further prove that,
		\begin{thm}
			 For any non-degenerate surfaces $\Sigma \in \mathcal{H}(6)$ with \(^{\#}\widetilde{\mathcal{J}}(\Sigma) < +\infty\), then at least three closed characteristics are elliptic, and among these three, at least two are irrational elliptic.
			%Suppose $^{\#}\widetilde{\mathcal{J}}(\Sigma) = 3$ for some non-degenerate $\Sigma \in \mathcal{H}(6)$, then all three closed characteristics on $\Sigma$ are elliptic.
			\label{thm:1-elliptic}
		\end{thm}
\begin{rem}
Based on the $n$-or-$\infty$ conjecture, when \(^{\#}\widetilde{\mathcal{J}}(\Sigma) < +\infty\), then \(^{\#}\widetilde{\mathcal{J}}(\Sigma)=3\), hence three elliptic closed characteristics is optimal.
\end{rem}
		
		%%我们的结果主要通过……
		Our results are mainly derived through analytical deduction by using of the Long-Zhu's common index jump methods \cite{lonzhu2002} and the estimation methods on iterative indices presented in the work of Hu-Ou \cite{huou2017}. In Section 2, we introduce the relevant concepts and formulas involved in this paper, and Section 3 provides a concise proof of the main Theorems.
		
		Let $\mathbb{N}, \mathbb{Z}, \mathbb{Q}, \mathbb{Q}^+, \mathbb{R},$ and $\mathbb{R}^+$ denote the sets of natural integers, integers, rational numbers, positive rational number, real numbers, and positive real numbers respectively. And let $(\cdotp , \cdotp )$, $| \cdotp |$, $\langle \cdot, \cdot \rangle$, $\|\cdot\|$ denote the standard inner product, the standard norm in $\mathbb{R}^{2n}$, the standard $L^2$ inner product and $L^2$ norm respectively. We also define the functions
		\begin{equation*}
		\begin{cases}
			{\lfloor a \rfloor} = \max\{k \in \mathbb{Z} \mid k \leq a\}, \quad &E(a) = \min\{k \in \mathbb{Z} \mid k \geq a\}, \\
			\{a\} = a - \lfloor a \rfloor, \quad &\varphi(a) = E(a) - \lfloor a \rfloor.
		\end{cases}
		\end{equation*}

	\section{ Review of the Maslov-type index theory}
	In this section, we briefly review the index theory for symplectic paths, for comprehensive details, please refer to \cite{lon1990}, \cite{lon2000} and \cite{lonzhu2002}.

	 \subsection{Definition of the Maslov-type index}
	As usual, let $\omega \in \mathbb{U}$, the symplectic group $\mathrm{Sp}(2n)$ is defined by
$$ \mathrm{Sp}(2n) = \{M\in \mathrm{GL}(2n,\mathbb{R})\,|\,M^TJM=J\}, $$
whose topology is induced from that of $\mathbb{R}^{4n^2}$. For $\tau>0$ we are interested
in paths in $\mathrm{Sp}(2n)$:
$$ \mathcal{P}_{\tau}(2n) = \{\gamma\in C([0,\tau],\mathrm{Sp}(2n))\,|\,\gamma(0)=I_{2n}\}, $$
which is equipped with the topology induced from that of $\mathrm{Sp}(2n)$. The
following real function was introduced in \cite{lon2000}:
$$ D_{\omega}(M) = (-1)^{n-1}\overline{\omega}^n\det(M-\omega I_{2n}), \ \
          \forall \omega\in\mathbb{U},\, M\in\mathrm{Sp}(2n). $$
Thus for any $\omega\in\mathbb{U}$ the following codimension one hypersurface in $\mathrm{Sp}(2n)$ is
defined in \cite{lon2000}:
$$ \mathrm{Sp}(2n)_{\omega}^0 = \{M\in\mathrm{Sp}(2n)\,|\, D_{\omega}(M)=0\}.  $$
For any $M\in \mathrm{Sp}(2n)_{\omega}^0$, we define a co-orientation of $\mathrm{Sp}(2n)_{\omega}^0$
at $M$ by the positive direction $\frac{d}{dt}Me^{t\epsilon J}|_{t=0}$ of
the path $Me^{t\epsilon J}$ with $0\leq t\leq1$ and $\epsilon>0$ being sufficiently
small. Let
\bea
\mathrm{Sp}(2n)_{\omega}^{\ast} &=& \mathrm{Sp}(2n)\setminus \mathrm{Sp}(2n)_{\omega}^0,   \nonumber\\
\mathcal{P}_{\tau,\omega}^{\ast}(2n) &=&
      \{\gamma\in\mathcal{P}_{\tau}(2n)\,|\,\gamma(\tau)\in\mathrm{Sp}(2n)_{\omega}^{\ast}\}, \nonumber\\
\mathcal{P}_{\tau,\omega}^0(2n) &=& \mathcal{P}_{\tau}(2n)\setminus  \mathcal{P}_{\tau,\omega}^{\ast}(2n).  \nonumber\eea
For any two continuous arcs $\xi$ and $\eta:[0,\tau]\to\mathrm{Sp}(2n)$ with
$\xi(\tau)=\eta(0)$, it is defined as usual:
$$
\eta\ast\xi(t) =\begin{cases}
		\xi(2t), & \text{if } 0\le t\le \tau/2, \\
		\eta(2t-\tau), & \text{if } \tau/2\le t\le \tau.
	\end{cases}\\
$$
Given any two $2m_k\times 2m_k$ matrices of square block form
$M_k=\left(
             \begin{array}{cc}
               A_k & B_k \\
               C_k & D_k \\
             \end{array}
           \right)
$
with $k=1, 2$,
as in \cite{lon2002}, the $\diamond$-product of $M_1$ and $M_2$ is defined by
the following $2(m_1+m_2)\times 2(m_1+m_2)$ matrix $M_1\diamond M_2$:
\begin{equation*} M_1\diamond M_2=\left(
                                    \begin{array}{cccc}
                                      A_1 & 0 & B_1 & 0 \\
                                      0 & A_2 & 0 & B_2 \\
                                      C_1 & 0 & D_1 & 0 \\
                                      0 & C_2 & 0 & D_2 \\
                                    \end{array}
                                  \right),
                               \end{equation*}  %\dm=\diamond
and denote by $M^{\diamond k}$ the $k$-fold $\diamond$-product $M\diamond\cdots\diamond M$.
Note that the $\diamond$-product of any two symplectic matrices is symplectic. For any two
paths $\gamma_j\in\mathcal{P}_{\tau}(2n_j)$ with $j=0$ and $1$, let
$\gamma_0\diamond\gamma_1(t)= \gamma_0(t)\diamond\gamma_1(t)$ for all $t\in [0,\tau]$.
%
%A special path $\xi_n\in \mathcal{P}_{\tau}(2n)$ is defined by
%\bea
%\xi_n(t) = \left(
%                \begin{array}{cc}
%                     2-\frac{t}{\tau} & 0 \\
%                     0 & (2-\frac{t}{\tau})^{-1} \\
%                   \end{array}
%                 \right)^{\diamond n}\ \ \text{ for}\;0\leq t\leq \tau.
%\eea

%and we can define the \(\omega\)-singular group \(\mathrm{Sp}(2n)_\omega^0\) and the \(\omega\)-regular group \(\mathrm{Sp}(2n)_\omega^*\) within the symplectic group \(\mathrm{Sp}(2n)\). Consider symplectic paths
%	\[
%	\mathcal{P}_\tau(2n) = \left\{ \gamma \in C([0, \tau], \mathrm{Sp}(2n)) \mid \gamma(0) = I \right\}
%	\]
%	 in \(\mathrm{Sp}(2n)\), and denote the \(\omega\)-singular path set and \(\omega\)-regular path set by \(\mathcal{P}_{\tau, \omega}^0(2n)\) and \(\mathcal{P}_{\tau, \omega}^*(2n)\), respectively.
For \(\gamma(\tau ) \in \mathrm{Sp}(2n)_\omega^*\), the \(\omega\)-index of \(\gamma\) can be defined by
\bea
	i_\omega(\gamma) = \left[ \mathrm{Sp}(2n)_\omega^0 : \gamma * \xi_n \right],\label{index function}
\eea
where \(
	\xi_n(t) = \begin{pmatrix} 2 - \frac{t}{\tau} & 0 \\ 0 & \left(2 - \frac{t}{\tau}\right)^{-1} \end{pmatrix}^{\diamond n}, 0 \leq t \leq \tau
	\) and the right hand side of (\ref{index function}) is the usual homotopy intersection
number, the orientation of $\gamma\ast\xi_n$ is its positive time direction under
homotopy with fixed end points.

	For \(\gamma(\tau) \in \mathrm{Sp}(2n)_\omega^0\), the \(\omega\)-index of \(\gamma\) is defined by
	$$
	i_\omega(\gamma) = \sup_{U \in \mathcal{F}(\gamma)} \inf \left\{ i_\omega(\beta) \mid \beta \in U \cap \mathcal{P}_{\tau, \omega}^*(2n) \right\},
	$$
	where \(\mathcal{F}(\gamma)\) denotes the set of open neighborhoods of \(\gamma\) in \(\mathcal{P}_\tau(2n)\). For any symplectic matrix \(M \in \mathrm{Sp}(2n)\), the \(\omega\)-nullity is
	\[
	\nu_\omega(M) = \dim_{\mathbb{C}} \ker_{\mathbb{C}} (M - \omega I).
	\]
	In particular, we set \(\nu_\omega(\gamma) := \nu_\omega(\gamma(\tau))\). The integer pair \((i_\omega(\gamma), \nu_\omega(\gamma)) \in \mathbb{Z} \times \{0, 1, \dots, 2n\}\) is called the index function of \(\gamma\) with respect to \(\omega\).

For any $\gamma\in\mathcal{P}_{\tau}(2n)$, For simplicity, define
$$(i(\gamma,k), \nu(\gamma,k)) = (i_1(\gamma^k), \nu_1(\gamma^k)), \ \ \forall k\in\mathbb{N}.
$$
where \((i_1(\gamma^k), \nu_1(\gamma^k))\) is the index function of the iterative symplectic path \(\gamma^k(t) := \gamma(t - j\tau) \gamma(\tau)^j\), here \(j\tau \leq t \leq (j+1)\tau\), \(j = 0, 1, \dots, k-1\).
The mean index $\hat{i}(\gamma,m)$ per $m\tau$ for $m\in\mathbb{N}$ is defined by
$$ \hat{i}(\gamma,m) = \lim_{k\to +\infty}\frac{i(\gamma,mk)}{k}. $$

For $\Sigma\in \mathcal{H}(2n)$ and $\alpha\in(1,2)$, let $(\tau,x)\in \mathcal{J}(\Sigma,\alpha)$. For simplicity, we define
\bea
S^{+}(x)&=&S^{+}_{\gamma_{x}(\tau)}(1),\nonumber\\
(i(x,m),\nu(x,m))&=&(i(\gamma_{x},m),\nu(\gamma_{x},m)),\nonumber\\
\hat{i}(x,m)&=&\hat{i}(\gamma_{x},m),\nonumber
\eea
For all $m\in \mathbb{N}$, where $\gamma_{x}$ is the associated symplectic path of $(\tau,x)$.	
	\subsection{Iterative formula of the Maslov-type index}\label{sec2.2}
	The splitting numbers of a symplectic matrix \( M \in \mathrm{Sp}(2n) \) with respect to \( \omega \in \mathbb{U} \) can be defined as
	\[
	S_M^\pm(\omega) = \lim_{\varepsilon \to 0^+} i_{\omega \exp(\pm\sqrt{-1}\varepsilon)}(\gamma) - i_\omega(\gamma).
	\]
    Clearly, if \( \omega \notin \sigma(M) \), then \( S_M^\pm(\omega) = 0 \). If \( \omega \in \sigma(M) \), since splitting numbers are symplectically additive and any symplectic matrix can be decomposed into the diamond product of basic normal forms, we present the splitting numbers of basic normal forms below.
	
	For the basic normal form
	\begin{equation*}
		\begin{aligned}
	&D(\lambda) = \begin{pmatrix} \lambda & 0 \\ 0 & \lambda^{-1} \end{pmatrix}, \quad \lambda = \pm 2;\ \
	N_1(\lambda, a) = \begin{pmatrix} \lambda & a \\ 0 & \lambda \end{pmatrix}, \quad \lambda = \pm 1, a = \pm 1, 0,\\
	&R(\theta) = \begin{pmatrix} \cos\theta & -\sin\theta \\ \sin\theta & \cos\theta \end{pmatrix},\ \
	N_2(\omega, b) = \begin{pmatrix} R(\theta) & b \\ 0 & R(\theta) \end{pmatrix}, \quad \theta \in (0, \pi) \cup (\pi, 2\pi),\\
	    \end{aligned}
	\end{equation*}
	where \( b = \begin{pmatrix} b_1 & b_2 \\ b_3 & b_4 \end{pmatrix} \), \( b_i \in \mathbb{R} \) and \( b_2 \neq b_3 \), \( \omega = e^{\sqrt{-1}\theta} \), we have
	\begin{equation*}
		\begin{aligned}\label{splitting sq}
	\left( S_{D(\lambda)}^+(e^{\sqrt{-1}\theta}), S_{D(\lambda)}^-(e^{\sqrt{-1}\theta}) \right)& = (0,0),\ \
   	\left( S_{R(\theta)}^+(e^{\sqrt{-1}\theta}), S_{R(\theta)}^-(e^{\sqrt{-1}\theta}) \right) = (0,1),\\
	\left( S_{N_1(1,a)}^+(1), S_{N_1(1,a)}^-(1) \right) &=
	\begin{cases}
		(1,1), & \text{if } a = 1 \ \text{or} \  0, \\
		(0,0), & \text{if } a = -1,
	\end{cases}\\
	\left( S_{N_1(-1,a)}^+(-1), S_{N_1(-1,a)}^-(-1) \right) &=
	\begin{cases}
		(1,1), & \text{if } a = -1 \ \text{or} \  0, \\
		(0,0), & \text{if } a = 1,
	\end{cases}\\
	\left( S_{N_2(\omega,b)}^+(e^{\sqrt{-1}\theta}), S_{N_2(\omega,b)}^-(e^{\sqrt{-1}\theta}) \right)& =
	\begin{cases}
		(1,1), & \text{if \( N_2(\omega, b) \) is nontrivial}, \\
		(0,0), & \text{if \( N_2(\omega, b) \) is trivial}.
	\end{cases}
	    \end{aligned}
	\end{equation*}
	 Here we call \( N_2(\omega, b) \) nontrivial if \( (b_2 - b_3)\sin\theta < 0 \) and trivial if \( (b_2 - b_3)\sin\theta > 0 \).
For any $M_i\in \emph{Sp}(2n_i)$ with $i=0$ and $1$, there holds
	\begin{equation*} S^{\pm}_{M_0\diamond M_1}(\omega) = S^{\pm}_{M_0}(\omega) + S^{\pm}_{M_1}(\omega),\ \
 \text{and}\ \ S^{\pm}_{M}(\omega)=S^{\mp}_{M}(\bar{\omega}),\ \ \forall\ \ \omega\in\mathbb{U}.	\end{equation*}
where $\bar{\omega}$ is the conjugate of $\omega$. %The $\diamond$-product of $M_1$ and $M_2$ is defined by
%the following $2(m_1+m_2)\times 2(m_1+m_2)$ matrix $M_1\diamond M_2$:
%\begin{equation*} M_1\diamond M_2=\left(
%                                    \begin{array}{cccc}
%                                      A_1 & 0 & B_1 & 0 \\
%                                      0 & A_2 & 0 & B_2 \\
%                                      C_1 & 0 & D_1 & 0 \\
%                                      0 & C_2 & 0 & D_2 \\
%                                    \end{array}
%                                  \right),
%                               \end{equation*}  %\dm=\diamond
%and denote by $M^{\diamond k}$ the $k$-fold $\diamond$-product $M\diamond\cdots\diamond M$.
Based on the splitting numbers, Long-Zhu \cite{lonzhu2002} presented the abstract precise iteration formula for the Maslov-type index:
	\begin{equation*}
	\begin{aligned}
		i(\gamma, m) &= m\left(i(\gamma, 1) + S_M^+(1) - C(M)\right) \\
		&\quad + 2 \sum_{\theta \in (0, 2\pi)} E\left(\frac{m\theta}{2\pi}\right) S_M^-\left(e^{\sqrt{-1}\theta}\right) - \left(S_M^+(1) + C(M)\right),
	\end{aligned}
	\end{equation*}
	where
	\(
	C(M) = \sum_{0 < \theta < 2\pi} S_M^-\left(e^{\sqrt{-1}\theta}\right),
	\)
	and $\gamma(\tau)$ is symplectically similar to the following form,
	\begin{equation}
		\begin{aligned}
			\gamma(\tau) &\simeq N_1(1,1)^{\diamond p_-} \diamond I_{2p_0} \diamond N_1(1,-1)^{\diamond p_+} \diamond N_1(-1,1)^{\diamond q_-} \diamond -I_{2q_0} \diamond N_1(-1,-1)^{\diamond q_+} \\
			&\quad \diamond R(\theta_1) \diamond \cdots \diamond R(\theta_r) \diamond N_2(\omega_1, u_1) \diamond \cdots \diamond N_2(\omega_{r_*}, u_{r_*}) \\
			&\quad \diamond N_2(\lambda_1, v_1) \diamond \cdots \diamond N_2(\lambda_{r_0}, v_{r_0}) \diamond M_k,
		\end{aligned}
		\label{formula:decomposition of matrices}
	\end{equation}
	where \( N_2(\omega_j, u_j) \) are nontrivial forms with \( \omega_j = e^{\sqrt{-1}\alpha_j} \), \( \alpha_j \in (0, \pi) \cup (\pi, 2\pi) \) and \( u_j = \begin{pmatrix} u_{j1} & u_{j2} \\ u_{j3} & u_{j4} \end{pmatrix} \in \mathbb{R}^{2 \times 2} \); \( N_2(\lambda_j, v_j) \) are trivial forms with \( \lambda_j = e^{\sqrt{-1}\beta_j} \), \( \beta_j \in (0, \pi) \cup (\pi, 2\pi) \) and \( v_j = \begin{pmatrix} \nu_{j1} & \nu_{j2} \\ \nu_{j3} & \nu_{j4} \end{pmatrix} \in \mathbb{R}^{2 \times 2} \); \( M_k = D(2)^{\diamond k} \) or \( D(-2) \diamond D(2)^{\diamond (k-1)} \); \( p_-, p_0, p_+, q_-, q_0, q_+, r, r_* \) and \( r_0 \) are non-negative integers.
Furthermore, we say that a normal form is a rational normal form if the rotation angle in the normal form is a rational multiple of $\pi$. Then the number of the rational normal form
in $\{R(\theta_{1}),\ldots,R(\theta_{r})\}$ is denoted by $\tilde{r}$. Similarly,
for set $\{N_{2}(\omega_{1},u_{1}),\ldots,N_{2}(\omega_{r_{*}},u_{r_{*}})\}$
and $\{N_{2}(\lambda_{1},\nu_{1}),\ldots,N_{2}(\lambda_{r_{0}},\nu_{r_{0}})\}$,
the number of rational normal form is
denoted by $\tilde{r}_{*}$ and $\tilde{r}_{0}$ respectively. These integers and real numbers are uniquely determined by \( \gamma(\tau) \).
	
	Then we can rewrite  the iteration formula as
	\begin{equation*}
		\begin{aligned}
			i(\gamma, m) &= m\left(i(\gamma, 1) + p_- + p_0 - r\right) + 2\sum_{l=1}^r E\left(\frac{m\theta_j}{2\pi}\right) - r - p_- - p_0 \\
			&\quad -\frac{1+(-1)^m}{2}(q_0 + q_+) + 2\left( \sum_{l=1}^{r_*} \varphi\left(\frac{m\alpha_j}{2\pi}\right) - r_* \right),
		\end{aligned}
		\label{formula:iteration index}
	\end{equation*}
	
	\begin{equation}
		\begin{aligned}
			\nu(\gamma, m) &= \nu(\gamma, 1) + \frac{1+(-1)^m}{2}(q_- + 2q_0 + q_+) + 2(r + r_* + r_0) \\
			&\quad - 2\left( \sum_{l=1}^r \varphi\left(\frac{m\theta_l}{2\pi}\right) + \sum_{l=1}^{r_*} \varphi\left(\frac{m\alpha_l}{2\pi}\right) + \sum_{l=1}^{r_0} \varphi\left(\frac{m\beta_l}{2\pi}\right) \right),
		\end{aligned}
		\label{formula:iteration nullity}
	\end{equation}
	
	\begin{equation}
		\hat{i}(\gamma, 1) = i(\gamma, 1) + p_- + p_0 - r + \sum_{l=1}^r \frac{\theta_l}{\pi},
		\label{formula:mean index}
	\end{equation}
	
	\begin{equation}
	S_M^+(1) = p_- + p_0,
	\label{formula:S_M^+}
    \end{equation}
	
	\begin{equation}
	C(M) = \sum_{0<\theta<2\pi} S_M^-(e^{\sqrt{-1}\theta}) = q_0 + q_+ + r + 2r_*.
	\label{formula:C(M)}
    \end{equation}
	
	\subsection{Common index jump theorem}
	%%有限族辛道路的公共选择定理
	We first introduce the common index jump theorem for finite families of symplectic paths, which can be found in Long-Zhu \cite{lonzhu2002}.
	%Long and Zhu[Long-Zhu(2001)] presented the following common selection theorem for finite families of symplectic paths.
		\begin{thm}
			Let $\gamma_k \in \mathcal{P}_{\tau_k}(2n)$ for $k = 1, \dots, q$ be a finite family of symplectic paths, and let $M_k = \gamma_k(\tau_k)$. Suppose $\delta \in \left(0, \frac{1}{2}\right)$ satisfies $\delta \max\limits_{1 \leq k \leq q} \mu_k < \frac{1}{2}$, where $\mu_k = \sum_{\theta \in (0, 2\pi)} S_{M_k}^-(e^{\sqrt{-1}\theta})$. Suppose that  $\hat{i}(\gamma_k, 1) > 0$ for any $k = 1, \dots, q$, then there exist infinitely many $(N, m_1, \dots, m_q) \in \mathbb{N}^{q+1}$ such that
			\[
			I(k, m_k) = N + \Delta_k,
			\]
for any $k = 1, \dots, q$, where $$I(k, m_k) = m_k\left(\hat{i}(\gamma_k, 1) + S_{M_k}^+(1) - C(M_k)\right) + \sum_{\theta \in (0, 2\pi)} E\left(\frac{m_k \theta}{\pi}\right)S_{M_k}^-(e^{\sqrt{-1}\theta}),\ \ \Delta_k = \sum_{0 < \left\{m_k \frac{\theta}{\pi}\right\} < \delta} S_{M_k}^-(e^{\sqrt{-1}\theta}).$$
Morover, we have
	\bea
			\min\left\{ \left\{ \frac{m_k \theta}{\pi} \right\}, 1 - \left\{ \frac{m_k \theta}{\pi} \right\} \right\} < \delta, \label{delta}
			\eea
			\begin{equation}
			m_k \frac{\theta}{\pi} \in \mathbb{N}, \quad \text{if}\ \ \frac{\theta}{\pi} \in \mathbb{Q},\label{3.40}
			\end{equation}
where $e^{\sqrt{-1}\theta}\in\sigma(M_{k})$, $\frac{\theta}{\pi}\in(0,2)$ and $\delta$ can be chosen as small as we want.
			\label{thm:index jump 1}
		\end{thm}
		According to \cite{lonzhu2002}, the existence of infinitely many integer tuples $(N, m_1, \dots, m_q) \in \mathbb{N}^{q+1}$ in Theorem \ref{thm:index jump 1} can be transformed into a dynamical system problem on the torus. Let
			\[
			h = q + \sum_{k=1}^q \mu_k
			\]
			and
			\bea
			\nu = \left( \frac{1}{\hat{i}(\gamma_{1, 1})}, \dots, \frac{1}{\hat{i}(\gamma_{q,1})}, \frac{\theta_{1, 1}}{\pi \hat{i}(\gamma_{1, 1})}, \frac{\theta_{1, 2}}{\pi \hat{i}(\gamma_{1, 1})}, \dots, \frac{\theta_{1, \mu_1}}{\pi \hat{i}(\gamma_{1, 1})}, \frac{\theta_{2, 1}}{\pi \hat{i}(\gamma_{2,1})}, \dots, \frac{\theta_{q, \mu_q}}{\pi \hat{i}(\gamma_{q,1})} \right) \in \mathbb{R}^h,\label{3.35}
			\eea
where $e^{\sqrt{-1}\theta_{i,j}}\in \sigma (M_{i})$
for $1\leq j\leq \mu_{i}$, then Theorem \ref{thm:index jump 1} is equivalent to find a vertex $$\chi = (\chi_1, \dots, \chi_q, \chi_{1,1}, \dots, \chi_{1,\mu_1}, \dots \chi_{q,1}, \dots, \chi_{q,\mu_q})$$ on the torus $[0,1]^h$ and infinitely many integers $N \in \mathbb{N}$ such that
			\begin{equation}
			\bigl| \{ N\nu \} - \chi \bigr| < \epsilon, \label{3.36}
			\end{equation}
			where $\epsilon \in (0, \min\{\delta, 1/3\})$, and $\chi_k, \chi_{k,j} = 0$ or $1$ for $k = 1, \dots, q$ and $j = 1, \dots, \mu_1, 2, \dots, \mu_2, \dots, \mu_q$.
In fact, such a vertex $\chi$ exists, as stated in Theorem \ref{thm:index jump 2}.
       	\begin{thm}(cf. Theorem 4.2 of \cite{lonzhu2002})
       		Let the closure of $\{ mv \mid m \in \mathbb{N} \}$ on $T^h = (\mathbb{R}/\mathbb{Z})^h$ be $H$, and let $\pi: \mathbb{R}^h \to T^h$ be the projection map. Let $V = T_0 \pi^{-1}H$ denote the tangent space of $\pi^{-1}H$ at the origin of $\mathbb{R}^h$, and define
       		\[
       		A(v) = V \setminus \bigcup_{v_k \in \mathbb{R} \setminus \mathbb{Q}} \left\{ x = (x_1, \dots, x_h) \in V \mid x_k = 0 \right\}.
       		\]
       		Define $\psi(x)$ such that $\psi(x) = 0$ if $x \geq 0$, and $\psi(x) = 1$ if $x < 0$. Then for any $a = (a_1, \dots, a_h) \in A(v)$, the vector $\chi(a) = \left( \psi(a_1), \dots, \psi(a_h) \right)$ satisfies
       		\[
       		\bigl| \{ Nv \} - \chi \bigr| < \epsilon
       		\]
       		for infinitely many $N \in \mathbb{N}$. Moreover, this set $A(v)$ possesses the property: If $v\in \mathbb{R}^{h}\setminus \mathbb{Q}^{h}$, then
$\dim V\geq1, 0\notin A(v)\subset V, A(v)=-A(v)$ and $A(v)$ is open in $V$.
       		\label{thm:index jump 2}
       	\end{thm}
\begin{rem}  Given $M_{0}\in \mathbb{N}$, by the proof of Theorem 4.1 of \cite{lonzhu2002}, we may
further require $M_{0}|N$ (since the closure of the set $\{\{Nv\} : N\in \mathbb{N},M_{0}|N\}$)
is still a closed additive subgroup of $T^{h}$ for some $h\in\mathbb{N}$. Then we can use the
step 2 in Theorem 4.1 of of \cite{lonzhu2002} to get $N$), hence in our choice of $(N,m_1,...,m_q)$
in Theorem \ref{thm:index jump 1}, we can choose $M_0$ good enough such that $N\in\mathbb{N}$ further
satisfies
\begin{equation}
\frac{N}{M\hat{i}(\gamma_{k},1)}\in \mathbb{Z}, \ \ \text{for}\ \forall\  \hat{i}(\gamma_{k},1)\in \mathbb{Q}, \ \
k\in\{1,\ldots,q\}.\label{3.47}
\end{equation}
Furthermore, from (\ref{3.36}), we get $\hat{i}(\gamma_{k},1)\in \mathbb{Q}$ implies $\chi_{k}(a)=\psi(a_{k})=0$.
\end{rem}  		
       		In the following, the Maslov-type index According to \cite{lonzhu2002}, for each closed characteristics \( (\tau, x) \in \widetilde{\mathcal{J}}(\Sigma, \alpha) \), $\gamma_x(\tau)$ is symplectic similar to a matrix with the form $N_1(1,1)\diamond M$ with
       		$M \in \mathrm{Sp}(2n - 2)$, $\hat{i}(x) > 2$ and there is a subset $\mathcal{V}_\infty(\Sigma, \alpha) = \{(\tau_j, x_j)\,|\,j = 1, \cdots, q\} \subset \widetilde{\mathcal{J}}(\Sigma, \alpha) $ such that Theorem \ref{thm:index jump 1} can be
       		applied to it. Thus, we can obtain a correspondence between closed characteristics and some given integers.
       		
       		\begin{thm}(cf.\cite{lonzhu2002})
       			For some $a \in A(v)$, let $(N, m_1, \cdots, m_q) \in \mathbb{N}^{q+1}$ be given in Theorem \ref{thm:index jump 1} and \ref{thm:index jump 2}. Define $\varrho_n(\Sigma): \mathcal{H}(2n) \to \mathbb{N} \cup \{+\infty\}$ by
       			\[
       			\varrho_n(\Sigma)=
       			\begin{cases}
       				+\infty, & \text{if } ^{\#}\mathcal{V}_\infty(\Sigma,\alpha)=+\infty, \\
       				\min\left\{ \frac{\hat{i}(x,1)+2S^+(x)-\nu(x,1)+n}{2} \mid (\tau,x) \in \mathcal{V}_\infty(\Sigma,\alpha) \right\}, & \text{if } ^{\#}\mathcal{V}_\infty(\Sigma,\alpha)<+\infty.
       			\end{cases}
       			\]
       	Then for $s=1, \ldots, \varrho_n(\Sigma)$, $x_{j(s)}$ are geometric different, thus there are at least $\varrho_n(\Sigma)\geq\lfloor\frac{n}{2}\rfloor+1$ closed characteristics. Moreover, at least $\varrho_n(\Sigma)-1\geq[\frac{n}{2}]$ among them have irrational mean index.
       Also, for each $s = 1, \ldots, \varrho_n(\Sigma)$, there exists a unique $j(s) \in \{1, \ldots, q\}$ and
       			an injection map $p: \mathbb{N} + K \to \mathcal{V}_\infty(\Sigma, \alpha) \times \mathbb{N}$ for some $K \geq 0$, such that
       			$p(N - s + 1) = ((\tau_{j(s)}, x_{j(s)}), 2m_{j(s)})$ and
       			\[
       			i(x_{j(s)}, 2m_{j(s)}) \leq 2N - 2s + n \leq i(x_{j(s)}, 2m_{j(s)}) + \nu(x_{j(s)}, 2m_{j(s)}) - 1.
       			\]
       			Moreover, for any $s_1, s_2 \in \{1, \ldots, \varrho_n(\Sigma)\}$ with $s_1 < s_2$, we have:
       			\[
       			\hat{i}(x_{j(s_2)}, 2m_{j(s_2)}) < \hat{i}(x_{j(s_1)}, 2m_{j(s_1)}).
       			\]
and if all the closed characteristics in non-degenerate, then $\varrho_n(\Sigma)=n$.
       			\label{thm:index jump 3}
       		\end{thm}
Based on this Theorem, we can further obtain following lemma.
\begin{lem}\label{lem choose a}
We can choose an $\hat{a}\in A(v)$ such that $\chi_{j(s)}(\hat{a})=1$ for some $s=1,...,\varrho_{n}(\Sigma)$.
\end{lem}
\begin{proof}
Take $a\in A(v)$, from the definition of $A(v)$,  one can see that for any $i=1,...h$, if $v_{i}\in\mathbb{R}\setminus \mathbb{Q}$, then $a_{i}\neq 0$.
Since $v$ is given by (\ref{3.35}),  then $v_{i}=1/\hat{i}(\gamma_{i, 1})$ for $i=1,...,q$, this combine with Theorem \ref{thm:index jump 3} tells
us that at least $\varrho_n(\Sigma)-1\geq[\frac{n}{2}]$ of $v_{j(s)}, s=1,...,\varrho_n(\Sigma)$ are irrational number. Therefore
there exist at least $\varrho_n(\Sigma)-1\geq[\frac{n}{2}]$ of $a_{j(s)}, s=1,...,\varrho_n(\Sigma)$ are non-zero. If $a_{j(s)}<0$ for some $s\in\{1,...,\varrho_n(\Sigma)\}$,
then $\chi_{j(s)}(a)=1$, take $\hat{a}=a$, such $\hat{a}$ is found. If $a_{j(s)}>0$ for some $s\in\{1,...,\varrho_n(\Sigma)\}$, since $A(v)=-A(v)$, we can take $\hat{a}=-a\in A(v)$, then $\hat{a}_{j(s)}<0$, hence $\chi_{j(s)}(\hat{a})=1$, this complete the proof.
\end{proof}
%\begin{lem}\label{lem choose a2}
%We can choose an $\hat{a}\in A(v)$ such that $\chi_{j(s)}(\hat{a})=0$ for some $s=1,...,\varrho_{n}(\Sigma)$.
%\end{lem}
%\begin{proof}
%The proof is similar to that of Lemma \ref{lem choose a} and is omitted here, leaving it to the reader.
%\end{proof}
\section{The elliptic closed characteristics}
Based on Theorems \ref{thm:index jump 1} and Theorem \ref{thm:index jump 3}, Hu-Ou \cite{huou2017} give
some inequalities below. These inequalities are useful in our proof of the main Theorems.
%	Before presenting the proof of Theorem \ref{thm:1-elliptic}, we first introduce the lemmas required for the argument, all of which were established in \cite{huou2017}.
		\begin{lem}(cf. Thm 2.7 of \cite{huou2017})
            For any $s \in \{1, \dots, \varrho_n(\Sigma)\}$, we have
			\begin{equation}
			2s \geq n + S_{M_{j(s)}}^+(1) + C(M_{j(s)}) - 2\Delta_{j(s)} - \nu(x_{j(s)}, 2m_{j(s)}) + 1,
			\label{formula:2s more}
			\end{equation}
			\begin{equation}
			2s \leq n + S_{M_{j(s)}}^+(1) + C(M_{j(s)}) - 2\Delta_{j(s)}.
			\label{formula:2s less}
		    \end{equation}
            \label{lem:Hu-Ou 1}
		\end{lem}

	\begin{lem}(cf. Cor 2.8 of \cite{huou2017})
			For any $s_1, s_2 \in \{1, \dots, \varrho_n(\Sigma)\}$ with $s_1 < s_2$,

i) If $D_{j(s_{2})}\in\mathbb{Q}$, then $\chi_{j(s_{2})}(a)=0$ and
\bea \nonumber
N=\left(\left[\frac{N}{MD_{j(s_{2})}}\right]+\chi_{j(s_{2})}(a)\right)MD_{j(s_{2})}
&<&\left(\left[\frac{N}{MD_{j(s_{1})}}\right]+\chi_{j(s_{1})}(a)\right)MD_{j(s_{1})}
\eea
with $D_{j(s_{1})}\in\mathbb{R}\setminus\mathbb{Q}$, $\chi_{j(s_{1})}(a)=1$.

ii) If $\chi_{j(s_{2})}(a)=1$, then $D_{j(s_{2})}\in\mathbb{R}\setminus\mathbb{Q}$ and
\bea \nonumber
N<\left(\left[\frac{N}{MD_{j(s_{2})}}\right]+\chi_{j(s_{2})}(a)\right)MD_{j(s_{2})}
&<&\left(\left[\frac{N}{MD_{j(s_{1})}}\right]+\chi_{j(s_{1})}(a)\right)MD_{j(s_{1})}
\eea
with $D_{j(s_{1})}\in\mathbb{R}\setminus\mathbb{Q}$, $\chi_{j(s_{1})}(a)=1$.

iii)  \bea  \chi_{j(s_{2})}(a)\leq \chi_{j(s_{1})}(a).\nonumber\eea
\label{lem:Hu-Ou 2}
	\end{lem}
Based on the $iii)$ of Lemma \ref{lem:Hu-Ou 2}, Lemma \ref{lem choose a}, we can easy to obtain below a simple but useful lemma.
\begin{lem}\label{lem choose a 2}
We can choose an $\hat{a}\in A(v)$ such that $\chi_{j(1)}(\hat{a})=1$.
\end{lem}
\begin{proof}
From Lemma \ref{lem choose a}, we can choose an $\hat{a}\in A(v)$ such that $\chi_{j(s)}(\hat{a})=1$ for some $s=1,...,\varrho_{n}(\Sigma)$.
Since $iii)$ of Lemma \ref{lem:Hu-Ou 2} implies that $1=\chi_{j(s)}(\hat{a})\leq \chi_{j(1)}(\hat{a})$, hence $\chi_{j(1)}(\hat{a})=1$.
\end{proof}
%\begin{lem}\label{lem choose a 3}
%We can choose an $\hat{a}\in A(v)$ such that $\chi_{j(\varrho_n(\Sigma))}(\hat{a})=0$.
%\end{lem}
%\begin{proof}
%From Lemma \ref{lem choose a2}, we can choose an $\hat{a}\in A(v)$ such that $\chi_{j(s)}(\hat{a})=0$ for some $s=1,...,\varrho_{n}(\Sigma)$.
%Since $iii)$ of Lemma \ref{lem:Hu-Ou 2} implies that $\chi_{j(\varrho_n(\Sigma))}(\hat{a})\leq \chi_{j(s)}(\hat{a})=0$, hence $\chi_{j(\varrho_n(\Sigma))}(\hat{a})=0$.
%\end{proof}
We must mention that the variable $\Delta_{j(s)}$ in Lemma \ref{lem:Hu-Ou 1} is important in the analysis of the ellipticity of the closed characteristics. In the following, we present an non-trivial estimation for variable $\Delta_{j(s)}$.
\begin{thm}
For any $s \in \{1, \dots, \varrho_n(\Sigma)\}$, we have following estimation
\bea
\Delta_{j(s)}\leq r-\tilde{r}+r_{*}-\tilde{r}_{*}.\label{inequal Delta1}
\eea
Moreover, if $\chi_{j(s)}(a) = 0$ and $\hat{i}(x_{j(s)}, 1) \in \mathbb{R}\setminus\mathbb{Q}$, then we have
\bea
\Delta_{j(s)}\leq r-\tilde{r}-1+r_{*}-\tilde{r}_{*}.\label{inequal Delta}
\eea
\label{thm delta1}
\end{thm}%\label{thm delta1}
\begin{proof}
From the properties of the splitting number in Section 2.2, for variable $\Delta_{j(s)}$, we have
\bea \nonumber
\Delta_{j(s)}&=&\sum_{\{m_{j(s)}\frac{\theta}{\pi}\}\in(0,\delta)}S_{M_{j(s)}}^{-}(e^{\sqrt{-1}\theta})\nonumber\\
&=&\sum_{\substack{l=1\\ \{m_{j(s)}\frac{\theta_{l}}{\pi}\}\in(0,\delta)}}^{r}S_{M_{j(s)}}^{-}(e^{\sqrt{-1}\theta_{l}})
+\sum_{\substack{l=1,\,\theta=\alpha_{l}\,\text{or}\, 2\pi-\alpha_{l},\\ \{m_{j(s)}\frac{\theta}{\pi}\}\in(0,\delta)}}^{r_{*}}S_{M_{j(s)}}^{-}(e^{\sqrt{-1}\theta}).\label{est sppliting2}
\eea
First, from (\ref{3.40}) in Theorem \ref{thm:index jump 1}, we have the trivial estimation
\bea
\sum_{\substack{l=1\\ \{m_{j(s)}\frac{\theta_{l}}{\pi}\}\in(0,\delta)}}^{r}S_{M_{j(s)}}^{-}(e^{\sqrt{-1}\theta_{l}})\leq r-\tilde{r}.\label{est sppliting4}
\eea
Moreover, one can see that if $\{m_{j(s)}\frac{\alpha_{l}}{\pi}\}\in(0,\delta)$, then $\{m_{j(s)}\frac{2\pi-\alpha_{l}}{\pi}\}=\{1-m_{j(s)}\frac{\alpha_{l}}{\pi}\}\in(1-\delta,1)$.
Similar, if $\{m_{j(s)}\frac{2\pi-\alpha_{l}}{\pi}\}\in(0,\delta)$, then $\{m_{j(s)}\frac{\alpha_{l}}{\pi}\}\in(1-\delta,1)$. Take $\delta$ small enough, only one of $\{m_{j(s)}\frac{\alpha_{l}}{\pi}\}$, $\{m_{j(s)}\frac{2\pi-\alpha_{l}}{\pi}\}$ belong to $(0,\delta)$, hence we have estimation
\bea
\sum_{\substack{l=1,\,\theta=\alpha_{l}\,\text{or}\, 2\pi-\alpha_{l},\\ \{m_{j(s)}\frac{\theta}{\pi}\}\in(0,\delta)}}^{r_{*}}S_{M_{j(s)}}^{-}(e^{\sqrt{-1}\theta})
\leq r_{*}-\tilde{r}_{*}.\label{est sppliting3}
\eea
Then inequalities (\ref{est sppliting4}), (\ref{est sppliting3}) and formula (\ref{est sppliting2}) implies (\ref{inequal Delta1}).

Now, assume that $\chi_{j(s)}(a)=0$ and $\hat{i}(x_{j(s)},1)\in\mathbb{R}\setminus\mathbb{Q}$, then (\ref{formula:mean index}) implies that
at least one of $\frac{\theta_{1}}{\pi}, \frac{\theta_{2}}{\pi},\ldots,\frac{\theta_{r}}{\pi}$
is irrational number, hence $r-\tilde{r}\geq1$ and
\bea
\left\{m_{j(s)}D_{j(s)}\right\}&=&
\left\{m_{j(s)}\left(i(x_{j(s)},1)+p_{-}+p_{0}-r+\sum_{l=1}^{r}\frac{\theta_{l}}{\pi}\right)\right\}\\ \nonumber
&=&\left\{m_{j(s)}\sum_{l=1,\, \frac{\theta_{l}}{\pi}\in \mathbb{R}\setminus\mathbb{Q}}^{r}\frac{\theta_{j}}{\pi}\right\}\\
&\leq&\sum_{l=1,\,\frac{\theta_{l}}{\pi}\in \mathbb{R}\setminus\mathbb{Q}}^{r}\left\{m_{j(s)}\frac{\theta_{j}}{\pi}\right\}, \label{4.7}
\eea
where $m_{j(s)}=([\frac{N}{MD_{j(s)}}]+\chi_{j(s)}(a))M=[\frac{N}{MD_{j(s)}}]M$. The second equality comes from
(\ref{3.40}) in Theorem 2.5.

On the other hand,
\bea
\left\{m_{j(s)}D_{j(s)}\right\}&=&
\left\{\left[\frac{N}{MD_{j(s)}}\right]MD_{j(s)}\right\}\nonumber\\
&=&\left\{N-\left\{\frac{N}{MD_{j(s)}}\right\}MD_{j(s)}\right\}, \label{4.8}
\eea
and from (\ref{3.35}), (\ref{3.36}) and $\chi_{j(s)}(a)=0$, we get that $\{\frac{N}{MD_{j(s)}}\}=
\left|\{\frac{N}{MD_{j(s)}}\}-\chi_{j(s)}(a)\right|<\epsilon$ for any given
$\epsilon$ small enough, take $0<\epsilon<\frac{1-r\delta}{MD_{j(s)}}$ with $\delta$ small enough,
where $\delta$ is given by (\ref{delta}) in Theorem \ref{thm:index jump 1}, then from (\ref{4.8}), we obtain
$\{m_{j(s)}D_{j(s)}\}>r\delta,$
this combines with (\ref{4.7}), we have
\bea
r\delta
<\sum_{l=1,\,\frac{\theta_{l}}{\pi}\in \mathbb{R}\setminus\mathbb{Q}}^{r}\left\{m_{j(s)}\frac{\theta_{j}}{\pi}\right\},
\eea
hence at least one of the elements in $\{\frac{\theta_{j}}{\pi}|\frac{\theta_{j}}{\pi}\in \mathbb{R}\setminus\mathbb{Q}, j=1,\ldots,r\}$
satisfies $\{m_{j(s)}\frac{\theta_{j}}{\pi}\}\not\in(0,\delta)$. This combine with the properties of splitting number in Section \ref{sec2.2}, we have
\bea
\sum_{l=1,\{m_{j(s)}\frac{\theta_{l}}{\pi}\}\in(0,\delta)}^{r}S_{M_{j(s)}}^{-}(e^{\sqrt{-1}\theta_{l}})\leq r-\tilde{r}-1,\label{est sppliting1}
\eea
Combine with (\ref{est sppliting1}), (\ref{est sppliting2}), (\ref{est sppliting4}) and (\ref{est sppliting3}), we the following estimation, this completes the proof,
\bes
\Delta_{j(s)}\leq r-\tilde{r}-1+r_{*}-\tilde{r}_{*}.
\ees
\end{proof}
Based on Theorem \ref{thm delta1} and inequalities \ref{formula:2s more} in Lemma \ref{lem:Hu-Ou 1}, we have following estimation.
\begin{cor}\label{cor 2s}
For any $s \in \{1, \dots, \varrho_n(\Sigma)\}$, we have
			\bea
			2s\geq p_{-}+q_{+}+2r_{*}+2(r_{0}-\tilde{r}_{0})+k+1. \label{inequal s1}
			\eea
Moreover, if $\chi_{j(s)}(a) = 0$ and $\hat{i}(x_{j(s)}, 1) \in \mathbb{R}\setminus\mathbb{Q}$, we further have
\bea
			2s\geq p_{-}+q_{+}+2r_{*}+2(r_{0}-\tilde{r}_{0})+k+3. \label{inequal s2}
			\eea
\end{cor}
\begin{proof}
Based on (\ref{3.40}) in Theorem \ref{thm:index jump 1}, we can rewrite formula (\ref{formula:iteration nullity}) by
\bes
\nu(\gamma_{j(s)},2m_{j(s)})
=p_{-}+2p_{0}+p_{+}+q_{-}+2q_{0}+q_{+}+2(\tilde{r}+\tilde{r}_{*}+\tilde{r}_{0}).
\ees
This combine with formula (\ref{formula:S_M^+}), (\ref{formula:C(M)}) and inequality (\ref{inequal Delta1}), (\ref{inequal Delta}), we obtain (\ref{inequal s1}), (\ref{inequal s2}).
\end{proof}
Using the above inequalities, we obtain the following important result.
\begin{thm}
	   The closed characteristics $x_{j(1)}, x_{j(2)}$ has following properties:
\vskip 0.2 cm
(1) $x_{j(1)}$ is an elliptic closed characteristic and
\begin{equation}
		\begin{aligned}
			\gamma_{j(1)}(\tau) &\simeq N_1(1,1)\diamond I_{2p_0} \diamond N_1(1,-1)^{\diamond p_+} \diamond N_1(-1,1)^{\diamond q_-} \diamond -I_{2q_0}\\
			&\ \ \diamond R(\theta_1) \diamond \cdots \diamond R(\theta_r)
			\diamond N_2(\lambda_1, v_1) \diamond \cdots \diamond N_2(\lambda_{r_0}, v_{r_0}),\label{normal form s=1,2}
		\end{aligned}
	\end{equation}
with $r_{0}=\tilde{r}_{0}$ (i.e.all $N_{2}(\lambda_{1},\nu_{1}),\ldots,N_{2}(\lambda_{r_{0}},\nu_{r_{0}})$ are rational normal form). Moreover, when $\Sigma$ is non-degenerate, then $x_{j(1)}$ is irrationally elliptic.
\vskip 0.2 cm
(2) \, $\hat{i}(x_{j(1)}, 1) \in \mathbb{R}\setminus\mathbb{Q} \Leftrightarrow \chi_{j(1)}(a) = 1$.
\vskip 0.2 cm
(3) If\, $\hat{i}(x_{j(2)}, 1) \in \mathbb{R}\setminus\mathbb{Q}$ and $\chi_{j(2)}(a) = 0$, then $x_{j(2)}$ is an elliptic closed characteristic and it admits a symplectic normal form of similar type (\ref{normal form s=1,2}). When $\Sigma$ is non-degenerate, then $x_{j(2)}$ is irrationally elliptic.
		\label{lem:Hu-Ou 3}
\end{thm}
\begin{proof}
For $(1)$, from the fact $p_{-}\geq1$ and inequality (\ref{inequal s1}) in Corollary \ref{cor 2s}, when $s=1$, we have
$$
2\geq p_{-}+q_{+}+2r_{*}+2(r_{0}-\tilde{r}_{0})+k+1\geq2+q_{+}+2r_{*}+2(r_{0}-\tilde{r}_{0})+k.
$$
hence $p_{-}=1, q_{+}=r_{*}=k=0$ and $r_{0}=\tilde{r}_{0}$, from the symplectic decomposition (\ref{formula:decomposition of matrices}), we get (\ref{normal form s=1,2}).
Moreover, when the compact convex surface $\Sigma$ is non-degenerate, we further have $p_0=p_+=q_-=q_0=\tilde{r}=\tilde{r}_0=0$, hence $x_{j(1)}$ is irrationally elliptic.

For $(2)$, if $\hat{i}(x_{j(1)}, 1) \in \mathbb{Q}$ then from $i)$ in Lemma \ref{lem:Hu-Ou 1}, we must have $\chi_{j(1)}(a) = 0$, hence  $\chi_{j(1)}(a) = 1$ implies
$\hat{i}(x_{j(1)}, 1) \in \mathbb{R}\setminus\mathbb{Q}$. Moreover, if\, $\hat{i}(x_{j(1)}, 1) \in \mathbb{R}\setminus\mathbb{Q}$ and $\chi_{j(1)}(a) = 0$, using inequality (\ref{inequal s2}) in Corollary \ref{cor 2s}, when $s=1$, we have
$$
2\geq p_{-}+q_{+}+2r_{*}+2(r_{0}-\tilde{r}_{0})+k+3\geq4+q_{+}+2r_{*}+2(r_{0}-\tilde{r}_{0})+k,
$$
which implies $2\geq 4$, contradiction. Therefore, $\hat{i}(x_{j(1)}, 1) \in \mathbb{R}\setminus\mathbb{Q}$ must implies $\chi_{j(1)}(a) = 1$.

For $(3)$, if $\chi_{j(2)}(a) = 0$ and $\hat{i}(x_{j(2)}, 1) \in \mathbb{R}\setminus\mathbb{Q}$, using inequality (\ref{inequal s2}) in Corollary \ref{cor 2s}, when $s=2$, we have
$$
4\geq p_{-}+q_{+}+2r_{*}+2(r_{0}-\tilde{r}_{0})+k+3\geq4+q_{+}+2r_{*}+2(r_{0}-\tilde{r}_{0})+k.
$$
hence $p_{-}=1, q_{+}=r_{*}=k=0$ and $r_{0}=\tilde{r}_{0}$, from the symplectic decomposition (\ref{formula:decomposition of matrices}), then $x_{j(2)}$ have a symplectic normal form of type (\ref{normal form s=1,2}). Similar, when the compact convex surface $\Sigma$ is non-degenerate, we further have $p_0=p_+=q_-=q_0=\tilde{r}=\tilde{r}_0=0$, hence $x_{j(2)}$ is irrationally elliptic.
\end{proof}
Based on the above Theorem \ref{lem:Hu-Ou 3}, we can easy to prove our main Theorem \ref{thm:2-irr-elliptic}.
\begin{proof}[\textbf{Proof of Theorem \ref{thm:2-irr-elliptic}.}]
For any fix $a\in A(v)$ and $s\in\{1,\ldots,\varrho_{n}(\Sigma)\}$, consider the inject map
$p(N-s+1)=([(\tau_{j(s)},x_{j(s)})], 2m_{j(s)})$ given by Theorem \ref{thm:index jump 3}. Now, from Lemma \ref{lem choose a 2}, we can choose an $\hat{a}\in A(v)$ such that $\chi_{j(1)}(\hat{a})=1$. Then Theorem \ref{lem:Hu-Ou 3}-$(1)$ implies $x_{j(1)}$ is an elliptic closed characteristic and
\begin{equation}
		\begin{aligned}
			\gamma_{j(1)}(\tau) &\simeq N_1(1,1)\diamond I_{2p_0} \diamond N_1(1,-1)^{\diamond p_+} \diamond N_1(-1,1)^{\diamond q_-} \diamond -I_{2q_0}\\
			&\ \ \diamond R(\theta_1) \diamond \cdots \diamond R(\theta_r)
			\diamond N_2(\lambda_1, v_1) \diamond \cdots \diamond N_2(\lambda_{r_0}, v_{r_0}),\nonumber
		\end{aligned}
	\end{equation}
with all $N_{2}(\lambda_{1},\nu_{1}),\ldots,N_{2}(\lambda_{r_{0}},\nu_{r_{0}})$ are rational normal form. Moreover, it is easy to see that $x_{j(1)}$ is irrationally elliptic when $\Sigma$ is non-degenerate. Since $\chi_{j(1)}(\hat{a})=1$, from Theorem \ref{lem:Hu-Ou 3}-$(2)$, we know $\hat{i}(x_{j(1)},1)\in\mathbb{R}\setminus\mathbb{Q}$.
Now, choose $\tilde{a}=-\hat{a}\in A(v)$, using Theorem \ref{thm:index jump 3} again,
we still have $(\tilde{N},\tilde{m}_{1},\ldots,\tilde{m}_{q})$
, $\tilde{j}(s)$ and inject map
$p(\tilde{N}-s+1)=([(\tau_{\tilde{j}(s)},x_{\tilde{j}(s)})],2\tilde{m}_{\tilde{j}(s)}), s\in\{1,\ldots,\varrho_{n}(\Sigma)\}$. Similarly, by applying Theorem \ref{lem:Hu-Ou 3}-$(1)$ again, we obtain $x_{\tilde{j}(1)}$ has the same properties as $x_{j(1)}$. If $\tilde{j}(1)= j(1)$, from the definition of
$\chi(a)$, we know that $\chi_{j(1)}(\hat{a})=1$ implies $\chi_{j(1)}(-\hat{a})=0$, hence
$\chi_{\tilde{j}(1)}(\tilde{a})=\chi_{j(1)}(-\hat{a})=0$, but $\hat{i}(x_{\tilde{j}(1)},1)=\hat{i}(x_{j(1)},1)\in\mathbb{R}\setminus\mathbb{Q}$, this contradict with
Theorem \ref{lem:Hu-Ou 3}-$(2)$, hence this case will not happen, we always have $\tilde{j}(1)\neq j(1)$, then $x_{j(1)}, x_{\tilde{j}(1)}$ are two different closed characteristics
that possess the desired properties as required in Theorem \ref{thm:2-irr-elliptic}. This completes the proof.
\end{proof}
Moreover, under \textbf{Assumption A}, by a similar proof of Theorem \ref{thm:2-irr-elliptic}, we can easily obtain Theorem \ref{thm:3-irr-elliptic}.
\begin{proof} [\textbf{Proof of Theorem \ref{thm:3-irr-elliptic}.}]
From Lemma \ref{lem choose a 2}, we can choose an $\hat{a}\in A(v)$ such that $\chi_{j(1)}(\hat{a})=1$, Assumption A tells us that $x_{j(1)}$ is irrational elliptic , then similar the proof of Theorem \ref{thm:2-irr-elliptic}, take $a=-\hat{a}$, we have inject map $p(\tilde{N}-s+1)=([(\tau_{\tilde{j}(s)},x_{\tilde{j}(s)})],2\tilde{m}_{\tilde{j}(s)}), s\in\{1,\ldots,\varrho_{n}(\Sigma)\}$, then $x_{\tilde{j}(1)}$
must distinct from $x_{j(1)}$, use Assumption A again, $x_{\tilde{j}(1)}$ is another irrational elliptic closed characteristics.
\end{proof}

In the following, we will prove main Theorem \ref{thm:1-elliptic}, that is the existence of three elliptic closed characteristics on any $C^2$ compact convex hypersurfaces $\Sigma\subset \mathbb{R}^{6}$.
First, We need the following useful result, which is a corollary of Theorem A as proved by \c{C}ineli-Ginzburg-G\"{u}rel-Mazzucchellid \cite{CGGM2026}.
\begin{cor}(cf. Theorem A of \cite{CGGM2026})
For any $C^2$ compact convex hypersurface $\Sigma\subset \mathbb{R}^{2n}$, if $\Sigma$ has a hyperbolic closed characteristics, then $\Sigma$ has infinitely many geometrically distinct closed characteristics.\label{cor Ginzburgb}
\end{cor}
\begin{rem}
The original Theorem A in \cite{CGGM2026} is more general, the authors consider the closed Reeb orbits on the $(2n-1)$-dimensional standard
contact sphere, under very mild dynamical convexity type assumptions, the presence of one hyperbolic closed orbit implies the existence of infinitely many simple closed Reeb
orbits. It well known that the Hamiltonian flow on the hypersurface $\Sigma$ is equivalent to the Reeb flow on a standard contact sphere. Moreover, when $\Sigma$ is convex, it satisfies
the dynamical convexity type assumptions in \cite{CGGM2026}, therefore the above result is just a simple corollary of Theorem A in \cite{CGGM2026}.
\end{rem}
Second, another key observation of ours is that when the surfaces is non-degenerate, then Theorem \ref{thm:index jump 3} tells us that $\varrho_{3}(\Sigma)=3$ and formulas $(\ref{formula:S_M^+})$ and $(\ref{formula:C(M)})$ becomes
$$
S_M^+(1) = p_-=1,\ \ C(M) = r + 2r_*.
$$
Moreover, Lemma \ref{lem:Hu-Ou 1} shows that for any $s \in \{1, \dots, \varrho_n(\Sigma)\}$, we have $\nu(x_{j(s)}, 2m_{j(s)})=1$, hence
the inequalities in Lemma \ref{lem:Hu-Ou 1} become equalities
\bea
2s&=&n + S_{M_{j(s)}}^+(1) + C(M_{j(s)}) - 2\Delta_{j(s)}\nonumber\\
&=&4 + r + 2r_* - 2\Delta_{j(s)}.\label{equla s}
\eea
Based on this equalities and Corollary \ref{cor Ginzburgb}, we can obtain three elliptic closed characteristics in the following.
\begin{proof}[\textbf{Proof of Theorem \ref{thm:1-elliptic}.}]
%Theorem \ref{thm:index jump 3} tells us that $\varrho_{3}(\Sigma)=3$, then from Lemma \ref{lem choose a 3}, we first choose an $\hat{a}\in A(v)$ such that $\chi_{j(3)}(\hat{a})=0$.
First, Theorem \ref{lem:Hu-Ou 3} implies $x_{j(1)}$ is irrationally elliptic. For $s=2$,  formula (\ref{equla s}) becomes $0=r + 2r_* - 2\Delta_{j(2)}$,
from Corollary \ref{cor Ginzburgb}, we know that $x_{j(2)}$ can not be hyperbolic. Since $\Sigma$ is non-degenerate, if $x_{j(2)}$ is not an elliptic closed characteristics,
from the symplectic decomposition (\ref{formula:decomposition of matrices}), the only possible case is
$$
\begin{aligned}
			\gamma_{j(2)}(\tau) &\simeq N_1(1,1)\diamond R(\theta_1)\diamond M_1,
		\end{aligned}
$$
then $r=1$, this contradict with $0=r + 2r_* - 2\Delta_{j(2)}$, hence $x_{j(2)}$ must be elliptic.
For $s=3$, formula (\ref{equla s}) becomes
$2=r + 2r_* - 2\Delta_{j(2)}\leq r + 2r_*$, hence $r=2$ or $r_{*}=1$, this implies $x_{j(3)}$ is elliptic. Therefore we have proved the existence of three elliptic closed characteristics.

Furthermore, Theorem \ref{thm:index jump 3} tells us that at least $\varrho_3(\Sigma)-1=2$ among $x_{j(1)}, x_{j(2)}$ and $x_{j(3)}$ have irrational mean index, hence
we obtain that one of $\hat{i}(x_{j(2)})$ and $\hat{i}(x_{j(3)})$ must be irrational. Without loss of generality, we assume that $\hat{i}(x_{j(2)})$ is irrational, since we have proved that $x_{j(2)}$ is elliptic, then from the mean index formula (\ref{formula:mean index}), the only possible case is
$$
\begin{aligned}
			\gamma_{j(2)}(\tau) &\simeq N_1(1,1)\diamond R(\theta_1)\diamond \hat{M},
		\end{aligned}
$$
where $\hat{M}=R(\theta_2), N(\pm1,a)$ with $a=\pm1, 0$. Since $\Sigma$ is non-degenerate, $\hat{M}$ must be $R(\theta_2)$ with $\frac{\theta_{2}}{\pi}\in\mathbb{R}\setminus\mathbb{Q}$, then $x_{j(2)}$  must be irrational elliptic. In summary, we obtain that $x_{j(1)}$ is irrational elliptic and either $x_{j(2)}$ or $x_{j(3)}$ must be irrational elliptic. This completes the proof.
\end{proof}
       %\endpf

     \textbf{Acknowledgments}. The authors would like to express the gratitude to Xijun Hu for helpful discussions, and particularly for his constructive guidance of the research direction.  This work is partially  supported by the Taishan Scholars Climbing Program of Shandong($\sharp$ TSPD20240802). Y.Ou is partially supported by NSFC ($\sharp$12371192), the Young Taishan Scholars Program of Shandong Province ($\sharp$ tsqn202312055), and the Qilu Young Scholar Program of Shandong University.

     %导入参考文献
     
	\end{document}